\documentclass{amsart}
\usepackage[]{fontenc}
\usepackage{amsthm}
\usepackage{setspace}
\usepackage{amssymb}
\usepackage{esint}
\makeatletter
\theoremstyle{plain}
\newtheorem{thm}{Theorem}[section]
  \theoremstyle{plain}
  \newtheorem{cor}[thm]{Corollary}
  \theoremstyle{definition}
  \newtheorem{problem}[thm]{Problem}
  \theoremstyle{definition}
  \newtheorem{defn}[thm]{Definition}
  \theoremstyle{plain}
  \newtheorem{lem}[thm]{Lemma}
  \theoremstyle{plain}
  \newtheorem{prop}[thm]{Proposition}

\DeclareMathOperator{\id}{id}
\DeclareMathOperator{\supp}{supp}
\DeclareMathOperator{\diff}{diff}
\DeclareMathOperator{\ldim}{\underline{\dim}}

\DeclareMathOperator*{\essinf}{essinf}
\newcommand{\sqr}{{}^{{}^{}_\square}}

\makeatother

\begin{document}

\title{Geometric rigidity of $\times m$ invariant measures }

\author{Michael Hochman}

\thanks{Research supported by NSF grant 0901534. }

\keywords{Measure rigidity, invariant measure, interval map, fractal geometry,
geometric measure theory, scenery flow}

\subjclass[2000]{37A10, 37E10, 37A45, 37C45, 28A80}

\maketitle
\markboth{Michael Hochman}{Geometric rigidity of $\times m$ invariant measures}
\begin{abstract}
Let $\mu$ be a probability measure on $[0,1]$ which is invariant
and ergodic for $T_{a}(x)=ax\bmod1$, and $0<\dim\mu<1$. Let $f$
be a local diffeomorphism on some open set. We show that if $E\subseteq\mathbb{R}$
and $(f\mu)|_{E}\sim\mu|_{E}$, then $f'(x)\in\{\pm a^{r}\,:\, r\in\mathbb{Q}\}$
at $\mu$-a.e. point $x\in f^{-1}E$. In particular, if $g$ is a
piecewise-analytic map preserving $\mu$ then there is an open $g$-invariant
set $U$ containing $\supp\mu$ such that $g|_{U}$ is piecewise-linear
with slopes which are rational powers of $a$. 

In a similar vein, for $\mu$ as above, if $b$ is another integer
and $a,b$ are not powers of a common integer, and if $\nu$ is a
$T_{b}$-invariant measure, then $f\mu\perp\nu$ for all local diffeomorphisms
$f$ of class $C^{2}$. This generalizes the Rudolph-Johnson Theorem
and shows that measure rigidity of $T_{a},T_{b}$ is a result not
of the structure of the abelian action, but rather of their smooth
conjugacy classes: if $U,V$ are maps of $\mathbb{R}/\mathbb{Z}$
which are $C^{2}$-conjugate to $T_{a},T_{b}$ then they have no common
measures of positive dimension which are ergodic for both.
\end{abstract}

\section{\label{sec:Introduction}Introduction}

\subsection{\label{sub:Background}Background}

The motivating problem of this paper is to understand, for  {}``structured''
Borel probability measures on $\mathbb{R}$, which transformations
can map one measure to another, in whole or in part, and how the structure
of the measures determines this. More precisely, writing $f\mu$ for
the measure $f\mu(A)=\mu(f^{-1}A)$, we ask what one can say about
maps $f$ for which is $f\mu$ non-singular with respect $\mu$, or
with respect to some other measure $\nu$. The expectation is that
highly structured measures should be preserved by a small number of
maps whose structure reflects that of the measure.

There are a few elementary things one can say. Since any two non-atomic
probability measures on $\mathbb{R}$ can be mapped to each other
by a continuous function, one must impose some regularity assumption
for the question to make sense. In this paper we consider diffeomorphisms
of the line, and denote the set of such maps by $\diff(\mathbb{R})$.
Write $\diff^{k}(\mathbb{R})$ for the set of diffeomorphisms of class
$C^{k}$. Also, there are  some trivial cases of measures $\mu$ such
that $\mu,f\mu$ are non-singular for many $f\in\diff(\mathbb{R})$,
namely if $\mu$ has a Lebesgue component or atoms. One must therefore
consider measures without such components. Recall that the local dimension
of a measure $\mu$ at $x$ is the limit\[
D(\mu,x)=\lim_{r\rightarrow0}\frac{\log\mu(B_{r}(x))}{\log r}\]
where $B_{r}(x)$ is the ball of radius $r$ around $x$. In general
the local dimension may not exist but it exists in the cases which
will interest us. We will work with measures of \emph{intermediate
dimension}, meaning that $0<D(\mu,x)<1$ at $\mu$-a.e. $x$. This
rules out Lebesgue component, and also any zero-dimensional component,
including atoms.

In this work we consider probability measures which are invariant
under the maps $T_{a}:[0,1]\rightarrow[0,1]$ defined for integers
$a\geq2$ by\[
T_{a}x=bx\bmod1\]
i.e. measures such that $T_{a}\mu=\mu$. We sometimes view $T_{a}$
as a map of the torus $\mathbb{R}/\mathbb{Z}$. We note that there
is a very rich supply of $T_{a}$-invariant measures, including some
self-similar measures but also many (most) which are not.

For $T_{a}$-invariant measures most work to date has focused on their
behavior for maps which are related to the group structure of $\mathbb{R}/\mathbb{Z}$,
i.e. when $f$ is one of the maps $T_{b}$ (an endomorphism of $\mathbb{R}/\mathbb{Z}$)
or translation by an element of $\mathbb{R}/\mathbb{Z}$. The principal
result for endomorphisms is the measure rigidity theorem of Rudolph
\cite{Rudolph90} and Johnson \cite{Johnson92}. Write $a\sim b$
if $a,b$ are powers of a common integer, i.e. $a=c^{k}$ and $b=c^{m}$
for some $c,k,m\in\mathbb{N}$, and otherwise write $a\not\sim b$.
The Rudolph-Johnson Theorem states that if $\mu$ is a $T_{a}$-invariant
measure whose ergodic components all have entropy strictly between
$0$ and $\log a$, then it is not preserved by $T_{b}$ for any $b\not\sim a$.
The result can be slightly improved, using a later result of Rudolph-Johnson
\cite{JohnsonRudolph95}, to conclude that $\mu\perp\nu$ for every
$T_{b}$-invariant measure $\nu$ of intermediate dimension. To date
this is essentially the best result towards Furstenberg's $\times2$,$\times3$
conjecture, which predicts that there should be no non-atomic measures
except Lebesgue which are jointly invariant under $T_{a}$ and $T_{b}$
for $a\not\sim b$. For a survey of related algebraic conjectures
and results see \cite{Lindenstrauss2005}. There are also strong rigidity
results for smooth actions of $\mathbb{Z}^{d}$ and $\mathbb{R}^{d}$,
see e.g. Kalinin, Katok and Hertz \cite{KalininKatokHertz2008}. 

For translations we are aware only of the work of Host \cite{Host95}.
Define $\mu$ to be conservative for a subgroup $\Lambda\subseteq\mathbb{R}/\mathbb{Z}$
if for every set $A$ with $\mu(A)>0$ there is some $0\neq r\in\Lambda$
such that $\mu(A\cap(A+r))>0$. For the groups $\Lambda=\mathbb{Z}[\frac{1}{b}]$
of $b$-adic rationals when $\gcd(a,b)=1$, or for the cyclic subgroup
generated by an element $r\in\mathbb{R}/\mathbb{Z}$ such that $\{T_{a}^{n}r\}_{n\in\mathbb{N}}$
is dense in $\mathbb{R}/\mathbb{Z}$, Host showed, using methods of
harmonic analysis, that the only conservative $T_{a}$-invariant measure
is Lebesgue measure. Note that these results do not require any assumption
about the dimension of the meeasure, but also they do not directly
relate to our question, since they do not say anything about preservation
of $\mu$ under a particular rotation. 

In another direction, the question we are interested in has been studied
in the fractal geometry literature for rather general functions $f$,
e.g. $C^{1}$ or bi-Lipschitz, but for restricted classes of measures
such as self-similar measures on attractors of iterated function systems.
In situations like these it has been shown that measures arising from
different parameters are mutually singular, and cannot be easily deformed
into each other, see for example \cite{Elekes09}. Related questions
for certain classes of Cantor sets are known, for instance see \cite{CooperPignataro1988,BedfordFisher97,LiFeng2004}.

\subsection{\label{sub:Statement-of-results}Statement of results}

Write $\mu\sim\nu$ and $\mu\perp\nu$ to indicate that the measures
are equivalent or singular, respectively, and let $\mu|_{E}$ denote
the restricted measure $\mu|_{E}(A)=\mu(A\cap E)$. Thus $\mu,\nu$
are non-singular if and only if $\mu|_{E}\sim\nu|_{E}$ for some set
$E$ with $\nu(E)>0$.
\begin{thm}
\label{thm:self-rigidity}Let $\mu$ be a $T_{a}$-ergodic measure
of intermediate dimension. Then there exists $n\in\mathbb{N}$ such
that if $f\in\diff(\mathbb{R})$ and $(f\mu)|_{E}\sim\mu|_{E}$, then
\[
f'(x)\in\{\pm a^{k/n}\,:\, k\in\mathbb{Z}\}\qquad\mbox{ for }\mu\mbox{-a.e. }x\in f^{-1}E\]
More generally, if $\nu$ is another $T_{a}$-ergodic measure and
$(f\mu)|_{E}\sim\nu|_{E}$ then there exists $t\in\mathbb{R}$ such
that \[
f'(x)\in\{\pm t\cdot a^{k/n}\,:\, k\in\mathbb{Z}\}\qquad\mbox{ for }\mu\mbox{-a.e. }x\in f^{-1}E\]

\end{thm}
While we have stated the result for diffeomorphisms of $\mathbb{R}$,
the result is of a local nature and immediately applies to partially
defined or piecewise diffeomorphisms. 

Theorem \ref{thm:self-rigidity} is close to optimal. One clearly
cannot hope to get information about $f$ except on the support of
$\mu|_{E}$. On the other hand, a $T_{a}$-invariant measure $\mu$
is also invariant for $T_{a^{n}}$ for every $n\in\mathbb{N}$ and
sometimes also for $n=1/m$ when $a^{1/m}\in\mathbb{N}$. Thus one
cannot expect that some intrinsic property of $\mu$ will encode $a$,
and the best one can hope for is a power of $a$. The ergodicity assumption
is necessary also: for example fix a $T_{2}$-invariant measure $\mu$
of intermediate dimension and form $\nu=\frac{1}{2}\mu+\frac{1}{2}T_{3}\mu$,
which is also $T_{2}$-invariant. It is easy to see that there is
a piecewise linear map $f$ with $f\nu,\nu$ non-singular, and with
slopes $2$ and $3$ on sets of positive $\nu$-measure.

We do not know whether it can happen in the theorem that $a^{1/n}$
is not an integer. We also do not know whether a version of the theorem
is true under Lipschitz (rather than differentiability) conditions
on $f$. 

Under mild additional assumptions, Theorem \ref{thm:self-rigidity}
implies that very few maps can preserve a $T_{a}$-ergodic measure
of intermediate dimension. 
\begin{cor}
\label{cor:analytic-rigidity}Let $\mu$ and $n$ be as in the theorem..
Then every piecewise analytic map of $[0,1]$ which preserves $\mu$
is piecewise linear on an open set $U$ containing the support of
$\mu$ and on $U$ has slopes of the form $\pm a^{k/n}$, $k\in\mathbb{N}$. 
\end{cor}

\begin{cor}
\label{cor:rigidity-with-global-support}If $\mu$ is a $T_{a}$-ergodic
measure on $\mathbb{R}/\mathbb{Z}$ which has intermediate dimension
and is globally supported, then every $C^{1}$-map which preserves
$\mu$ has the form $T_{a'}$ for an integer $a'\sim a$.
\end{cor}
Easy examples show that there can be piecewise linear maps other than
$T_{a}$ which preserve $\mu$, for example one can easily construct
them by hand for the ''uniform'' measure on the middle third Cantor
set (i.e. Hausdorff measure at the appropriate dimension, normalized
to mass 1), which is $T_{3}$-invariant. 

A special case of the above we recover the Rudolph-Johnson Theorem
\cite{Rudolph90,Johnson92}. One may speculate that Theorem \ref{thm:self-rigidity}
holds without the entropy assumption, but proving this would imply
the full $\times2$,$\times3$ conjecture.

Our methods also allow us to generalize the Rudolph-Johnson theorem
in other ways.:
\begin{thm}
\label{thm:smooth-rigidity}If $a\not\sim b$ and $\mu,\nu$ are respectively
$T_{a}$ and $T_{b}$-ergodic measures of intermediate dimension,
then $f\mu\perp\nu$ for every $f\in\diff^{2}(\mathbb{R})$.
\end{thm}
We can eliminate the ergodicity and regularity assumptions under some
(rather weak) additional hypotheses, for example no ergodicity is
needed if $f$ is affine, or when when $f\in\diff(\mathbb{R})$ but
the ergodic components of $\nu$ under $T_{b}$ do not have spectrum
of the form $\frac{n}{\log a}$. 

An interesting consequence of the theorem above is that the measure
rigidity phenomenon in the Rudolph-Johnson Theorem is not a consequence
of properties of the abelian action generated by $T_{a}$ and $T_{b}$,
but rather of the smooth conjugacy classes of the individual maps
$T_{a},T_{b}$:
\begin{cor}
\label{cor:conjugates-rigidity}Let $a\not\sim b$ and let $f,g$
be self-maps of $\mathbb{R}/\mathbb{Z}$ which are (separately) $C^{2}$-conjugate
to $T_{a},T_{b}$, respectively. Then there is no measure of positive
Hausdorff dimension which is ergodic for both $f$ and $g$, except
possibly one which is equivalent to Lebesgue, and this occurs precisely
when the conjugating maps differ by a rotation.
\end{cor}
Note that for $f,g$ as above there will generally be no invariant
measures at all, but it is hard to verify this for any particular
pair of conjugates. It is known that if $f,g$ commute then they are
simultaneously $C^{0}$-conjugate to $T_{a},T_{b}$ \cite{JohnsonRudolph92},
and then Corollary \ref{cor:conjugates-rigidity} follows from the
Rudolph-Johnson Theorem.

After this paper was completed P. Shmerkin suggested another approach
which proves Theorem \ref{thm:smooth-rigidity} for $f\in\diff(\mathbb{R})$
and non-ergodic measures, but which does not give any version of Theorem
\ref{thm:self-rigidity}. This will appear elsewhere.

\subsection{\label{sub:methods}Methods}

To arrive at these results we study measures on $\mathbb{R}^{d}$
through the dynamics of the 1-parameter, measure-valued family obtained
by {}``zooming in'' on typical points for the measure. These families
are called sceneries, and for the measures we are considering they
behave like generic orbits in an appropriate dynamical system. Many
variants of the notion of a scenery have appeared in the fractal geometry
literature, see e.g. Bedford and Fisher \cite{BedfordFisher97}, and
have been used as a technical tool in the study of $T_{m}$-invariant
measures, in disguised form in Furstenberg's paper \cite{Furstenberg70}
and more recently in \cite{HochmanShmerkin09}. Our definition of
the scenery flow follows that of Gavish \cite{Gavish09}. A systematic
study of this notion and related ones can be found in \cite{Hochman09}.

Briefly, we show that for a $T_{a}$-invariant measure the sceneries
equidistribute for an ergodic flow whose pure-point spectrum contains
a rational multiple of $\frac{1}{\log a}$, and the remainig spectrum
comes from the original dynamics of $T_{a}$. Since these flows are
associated to measures in a geometric way, they are invariants of
the measure under the application of differentiable, locally bijective
maps. Furthermore, the flows derived at different points of a $T_{a}$-invariant
measure may have different phases with respect eigenvalues of the
form $\frac{n}{\log a}$, and applying a smooth map shifts the phase
by the logarithm of the derivative. The behavior of these phases underlies
the proof of Theorems \ref{thm:self-rigidity} and \ref{thm:smooth-rigidity}. 

This method of proof gives an effective condition for a measure $\mu$
to be a smooth image of a $T_{a}$-invariant measure, namely the spectrum
of the associated flow must contain a rational multiple of $\log a$.
In a sense, this explains how $\mu$ {}``encodes'' the arithmetic
class of the dynamics which generated it.

\subsection{\label{sub:Related-questions}Related questions}

We end this introduction with some open questions. Let us begin by
pointing out a connection between Theorem \ref{thm:smooth-rigidity}
and another conjecture of Furstenberg \cite{Furstenberg70}: If $X,Y\subseteq[0,1]$
are closed and invariant, respectively, under $T_{a}$ and $T_{b}$
for $a\not\sim b$, then for every affine map $f(x)=ux+v$,\[
\dim(X\cap fY)\leq\max\{0,\dim X+\dim Y-1\}\]
This says that all affine images of $Y$ should intersect $X$ in
as small a set as possible. Theorem \ref{thm:smooth-rigidity} gives
an analog of this for measures, though of course singularity of measures
implies nothing about the intersection of their topological supports.
On the other hand, note that Theorem \ref{thm:smooth-rigidity} has
content even when both $\mu,\nu$ are globally supported. 

Returning to the $\times2,\times3$ conjecture, the topological version
was proved by Furstenberg with no entropy assumptions: any closed
infinite subsect of $[0,1]$ which is invariatn under $T_{a}$,$T_{b}$
for $a\not\sim b$ is the entire interval. One may similarly ask for
topological versions of our results:
\begin{problem}
Suppose $A\subseteq[0,1]$ is an infinite, proper closed $T_{a}$-invariant
subset. If $f\in\diff^{1}(\mathbb{R})$ preserves $A$, must $|f'(x)|\in\{a^{r}\,:\, r\in\mathbb{Q}\}$
for all non-isolated $x\in A$? 
\end{problem}

\begin{problem}
\label{pro:topological-geometric-rigidity-1}Let $a\not\sim b$ and
let $f,g$ be maps of $\mathbb{R}/\mathbb{Z}$ which are (separately)
conjugate, respectively, to $T_{a},T_{b}$. If the conjugating maps
are sufficiently smooth, can we conclude that there are no infinite,
closed proper subsets of $\mathbb{R}/\mathbb{Z}$ which are jointly
$f$- and $g$-invariant? 
\end{problem}
We do not have answers except when for some $s>0$ the $s$-dimensional
Hausdorff measure is positive and finite on $A$. Then one can apply
our results to this measure.

In another direction, there is a strengthening of the Rudolph-Johnson
Theorem due Host \cite{Host95}, which assets that for $\gcd(a,b)=1$,
if $\mu$ is $T_{a}$-invariant of intermediate dimension, then $\mu$-a.e.
point $x$ equidistributes for Lebesgue measure under the action of
$T_{b}$ (in this case $x$ is said to be \emph{normal }in base $b$).
It is natural to ask whether the same is true when the measure is
distorted by a nice enough map. Let us formulate this question in
the simplest and most plausible case: 
\begin{problem}
Let $\mu$ be a $T_{a}$-invariant measure and of intermediate dimension,
and $\gcd(a,b)=1$. If $f(x)=ux+v$, $u\neq0$, is $f\mu$-a.e. point
normal in base $b$?

Finally, it is very likely that analogs of Theorems \ref{thm:self-rigidity}
and \ref{thm:smooth-rigidity} ours hold for more general interval
maps and in higher dimensional settings, for instance for measures
on the torus which are invariant under hyperbolic automorphisms under
suitable assumptions. However, in neither case does it appear that
our methods apply directly.
\end{problem}

\subsection{Organization}

In the next section we describe our methods in more detail. Proofs
are given in the following sections. We assume familiarity with basic
notions in ergodic theory, recalling some definitions as we go; for
an introduction see \cite{Walters82}. For background on geometric
measure theory see \cite{Mattila95}.

\subsection{Acknowledgment}

I am greatful to Jean Bourgain, Elon Lindenstrauss and Pablo Shmerkin
for their comments.

\section{\label{sec:Outline-of-the-argument}Main elements of the proofs}

In this section we give our main definitions and technical results,
and derive the main theorems from them. The remaining proofs are provided
in the next section.

\subsection{\label{sub:The-scenery-flow}The scenery flow}

Let $\mathcal{M}=\mathcal{M}_{d}$ denote the space of Radon measures
on $\mathbb{R}^{d}$, endowed with the weak topology. We use the term
\emph{measure }for Radon measures on $\mathbb{R}^{d}$, and denote
measures by $\mu,\nu,\sigma,\tau$ etc. We reserve the term \emph{distribution
}for Borel probability measures on $\mathcal{M}$, which we denote
by $P,Q,R$ etc. The space of distributions carries a measurable structure
defined by declaring the map $\mu\mapsto\mu(A)$ to be measurable
for all Borel sets $A\subseteq\mathcal{M}$. Write $\lambda$ for
Lebesgue measure and $\delta_{z}$ for the point mass at $z$, which
is a measure when $z\in\mathbb{R}^{d}$ and a distribution when $z\in\mathcal{M}$.
Let $\supp\mu$ denote the topological support of a measure, that
is, the complement of the union of all open sets of $\mu$-measure
zero. We write $\sim$ for equivalence of measures or distributions,
and also write $z\sim\mu$ to indicate that $z$ is distributed according
to $\mu$; which is intended will be clear from the context. 

For $x\in\mathbb{R}^{d}$ let $U_{x}:\mathbb{R}^{d}\rightarrow\mathbb{R}^{d}$
denote the translation map \[
U_{x}(y)=y-x\]
 and for $t\in\mathbb{R}$ let $S_{t}:\mathbb{R}^{d}\rightarrow\mathbb{R}^{d}$
denote the scaling map \[
S_{t}(x)=e^{t}x\]
Note the exponential time scale, which makes $S=(S_{t})_{t\in\mathbb{R}}$
into an action of the additive group $\mathbb{R}$ on $\mathbb{R}^{d}$.
These operations induce maps on $\mathcal{M}$: for $\mu\in\mathcal{M}$
we have $U_{x}\mu(A)=\mu(A+x)$ and $S_{t}\mu(A)=\mu(e^{-t}A)$. 

Let $\mu\mapsto\mu^{\sqr}$ denote the operation of normalizing a
measure to have unit mass on $[-1,1]^{d}$ and restricting it to this
cube, i.e.\[
\mu^{\sqr}=\frac{1}{\mu([-1,1]^{d})}\mu|_{[-1,1]^{d}}\]
Let\[
\mathcal{M}^{\sqr}=\{\mbox{Probability measures on }\mathbb{R}\mbox{ supported on }[-1,1]^{d}\}\]
and define $S_{t}^{\sqr}:\mathcal{M}^{\sqr}\rightarrow\mathcal{M}^{\sqr}$
by \[
S_{t}^{\sqr}\mu=(S_{t}\mu)^{\sqr}\]
so that $S^{\sqr}=(S_{t}^{\sqr})_{t\geq0}$ is a semigroup acting
on the set of $\mu\in\mathcal{M}^{\sqr}$ with $0\in\supp\mu$. This
is a Borel subset of $\mathcal{M}$, and the action is Borel, though
not continuous (it is discontinuous at measures which give positive
mass to the boundary of $[-1,1]^{d}$).
\begin{defn}
\label{def:scenery}Let $\mu\in\mathcal{M}$ and $x\in\supp\mu$.
The \emph{scenery }of $\mu$ at $x$ is the orbit of $U_{x}\mu$ under
$S^{\sqr}$, i.e. the one-parameter family \[
\mu_{x,t}=S_{t}^{\sqr}(U_{x}\mu)\qquad t\in\mathbb{R}^{+}\]

\end{defn}
In other words, the scenery is what one sees when {}``zooming in''
to $\mu$ at $x$, restricting and normalizing the measure as we go. 

In order to discuss the limiting behavior of the scenery, note that
$\mathcal{M}^{\sqr}$ may be identified with the weak-{*} compact
set of probability measures on $[-1,1]^{d}$. Thus we may speak of
convergence of distributions on $\mathcal{M}^{\sqr}$.
\begin{defn}
\label{def:scenery-flow} A measure $\mu\in\mathcal{M}$ \emph{generates
}a distribution $P$ at $x\in\supp\mu$ if the scenery $(\mu_{x,t})_{t\geq0}$
equidistributes for $P$, that is, if the uniform measure on the path
$(\mu_{x,t})_{0\leq t\leq T}$ converges weak-{*} to $P$ as $T\rightarrow\infty$.
Equivalently, for every $f\in C(\mathcal{M}^{\sqr})$,\[
\lim_{T\rightarrow\infty}\frac{1}{T}\int_{0}^{T}f(\mu_{x,t})dt=\int fdP\]
In this case $(\mathcal{M}^{\sqr},P,S^{\sqr})$ is called the \emph{scenery
flow} of $\mu$ at $x$.
\end{defn}
For a discussion of the properties of distributions generated in this
way see \cite{Hochman09}. We mention a few basic facts. First, if
there is a positive $\mu$-measure of points at which $\mu$ generates
some distribution (which may vary from point to point), then $\mu$-a.e.
one of the distributions is $S^{\sqr}$-invariant.%
\footnote{Note that $S^{\sqr}$ acts discontinuously, so this is not a complete
triviality.%
} Second, standard density arguments show that if $\nu\ll\mu$ then
a $\nu$-typical point generates a distribution for $\nu$ if and
only if it does for $\mu$, and in this case  the distributions are
the same. This applies in particular when $\nu=\mu|_{A}$.

The following simple observation is a key ingredient in our arguments.
Let $\diff_{+}(\mathbb{R})\subseteq\diff(\mathbb{R})$ denote the
subgroup of orientation-preserving maps. 
\begin{lem}
\label{lem:diffeo-shifts-sceneries}Let $\mu\in\mathcal{M}(\mathbb{R})$,
$x\in\supp\mu$ and $f\in\diff_{+}^{1}(\mathbb{R})$. Then after a
time-shift of $s=\log f'(x)$ the sceneries $\mu_{x,t}$ and $(f\mu)_{f(x),t}$
are asymptotic, i.e. \[
\lim_{t\rightarrow\infty}(\mu_{x,t}-(f\mu)_{f(x),t-s})=0\qquad\mbox{(weak-*)}\]
In particular $\mu$ generates $P$ at $x$ if and only if $f\mu$
generates $P$ at $f(x)$. 
\end{lem}
The proof is immediate from the fact that, locally, $f$ acts like
$S_{s}$ near $x$. 

The assumption that $f$ preserve orientation is necessary for the
conclusion that the scenery flows are the same, but if $f$ is orientation
reversing then the scenery flows are isomorphic as measure preserving
flows by way of the map induced on $\mathcal{M}^{\sqr}$ from $x\mapsto-x$.
We omit the details.

\subsection{\label{sub:scenery-flow-of-invariant-measures}The scenery flow of
$T_{a}$-invariant measures and its spectral properties}

Next we describe the scenery flow of a $T_{a}$-invariant measure
and, more generally, products of such measures. Recall that the diagonal
action of $T_{a}$ on $[0,1]^{d}$ is given by $T_{a}(x)=(T_{a}x_{1},T_{a}x_{2},\ldots,T_{a}x_{d})$.
A product of $T_{a}$-invariant measures is invariant under the diagonal
action. 

We require two more standard constructions, which we recall briefly.
First, the natural extension of an ergodic system $(\Omega,\nu,T)$
is an invertible ergodic system $(\widetilde{\Omega},\widetilde{\nu},\widetilde{T})$
factoring onto $(\Omega,\nu,T)$ and characterized by the property
that every factor map from an invertible system to $(\Omega,\nu,T)$
factors through $(\widetilde{\Omega},\widetilde{\nu},\widetilde{T})$.
The natural extension may be realized as the inverse limit of the
diagram of factor maps \[
\ldots\rightarrow(\Omega,\nu,T)\xrightarrow{T}(\Omega,\nu,T)\xrightarrow{T}(\Omega,\nu,T)\]
See also Section \ref{sec:The-scenery-process}. 

Second, the $t_{0}$-suspension of the discrete time system $(\Omega,\nu,T)$
is the flow defined on $\Omega\times[0,t_{0}]$ by \[
T_{t}(\omega,s)=(T^{[s/t_{0}]}\omega,\{\frac{s+t}{t_{0}}\}t_{0})\]
where $[r]$ and $\{r\}$ are the integer and fractional parts of
$r$, respectively. This flow preserves the product measure $\nu\times\frac{1}{t_{0}}\lambda|_{[0,t_{0})}$.
\begin{prop}
\label{pro:identification of generated-distributions-1}Let $\mu_{1},\ldots,\mu_{d}$
be $T_{a}$-invariant measures and $\mu=\times_{i=1}^{d}\mu_{i}$.
Then $\mu$ generates an $S^{\sqr}$-ergodic distribution $P_{x}$
at a.e. point $x$, and the system $(\mathcal{M}^{\sqr},P_{x},S^{\sqr})$
arises as a factor of the $\log b$-suspension of the ergodic component
$\mu^{(x)}$ of $x$ in $([0,1]^{d},\mu,T_{a})$. In particular $P_{x}$
depends only on the ergodic component $\mu^{(x)}$ of $x$. Furthermore,
if $\mu^{(x)}$ has intermediate dimension then $P_{x}$ is supported
on measures of intermediate dimension.
\end{prop}
This is proved in Section \ref{sec:The-scenery-process}. 

Write $e(t)=\exp(2\pi it)$. Recall that $\alpha\in\mathbb{R}$ is
an eigenvalue of an ergodic measure preserving system $(\Omega,\mathcal{B},\nu,T)$
if there is a complex function $\varphi\in L^{2}$ such that $\varphi\circ T=e(\alpha)\varphi$,
and $\alpha$ is an eigenvalue of a measure-preserving flow $(\Omega,\mathcal{B},\nu,(T_{t})_{t\in\mathbb{R}})$
if there is a function $\varphi\in L^{2}$ such that $\varphi\circ T_{t}=e(\alpha t)\varphi$
for all $t\in\mathbb{R}$. Suh $\varphi$ are called eigenfunctions,
and for ergodic transformations and flows they a.s. have constant
modulus, which we shall always assume has bees normalized to modulus
$1$. We denote the set of eigenvalues by $\Sigma$, with subscripts
to indicate the system in question.
\begin{thm}
\label{thm:identification-of-spectrum}Let $\mu$ be $T_{a}$-invariant
with intermediate entropy. Let $\Sigma_{\mu^{(x)}}$ denote the spectrum
of the ergodic component $\mu^{(x)}$ of $([0,1],\mu,T_{a})$ to which
$x$ belongs. Let $P_{x}$ be the distribution generated at $x$ and
$\Sigma_{P_{x}}$ the spectrum of $(\mathcal{M}^{\square},P_{x},S^{\square})$.
Then there is an $n\in\mathbb{N}$ such that \[
\frac{n}{\log a}\mathbb{Z}\subseteq\Sigma_{P_{x}}\subseteq\frac{1}{\log a}\Sigma_{\mu^{(x)}}\cup\frac{n}{\log a}\mathbb{Z}\]

\end{thm}
One cannot assume that $n=1$ since if $\mu$ is $T_{a}$ invariant
then for every $n$ it is also $T_{a^{n}}$ invariant, since $T_{a^{n}}=T_{a}^{n}$.
\begin{proof}
[Proof of Theorem \ref{thm:smooth-rigidity} under spectral assumptions]
Suppose $f\in\diff^{1}$ and $\mu,\nu$ are respectively $T_{a},T_{b}$
invariant, have intermediate entropy, and the ergodic components of
$\nu$ do not have pure point spectrum of the form $\frac{n}{\log a}$,
$n\in\mathbb{Z}$. By the theorem above the scenery flows generated
a.e. by $\mu$ have pure point spectrum of this form. If $f\mu\not\perp\nu$
then by Lemma \ref{lem:diffeo-shifts-sceneries}, with positive $\mu$-probability
the scenery flow generated by $\mu$ at $x$ is isomorphic to the
one generated by $\nu$ at $f(x)$. These possibilities are incompatible.
\end{proof}

\subsection{\label{sub:The-distribution-of-phases}The distribution of phases}

More refined information can be obtained from the distribution of
phases of the eigenfunctions of the scenery flow. That is, for a $T_{a}$-ergodic
measure and typical points $x,y$, we may consider the sceneries at
$x$ and $y$ and compare the relative phase of the eigenfunctions
corresponding to $\alpha=k/\log a$. 

Recall that a joining of $S^{\sqr}$-invariant distributions $P_{1},P_{2}$
is a distribution $P$ on $\mathcal{M}^{\sqr}\times\mathcal{M}^{\sqr}$
which projects to $P_{i}$ on the $i$-th coordinate, and which is
invariant under the diagonal flow $S^{\sqr}$ given by $S_{t}^{\sqr}(\mu,\nu)=(S_{t}^{\sqr}\mu,S_{t}^{\sqr}\nu)$.
When $P_{1},P_{2}$ are ergodic, the ergodic components of a joining
of $P_{1},P_{2}$ are also joinings of $P_{1},P_{2}$. A $P$-joining
is a joining of $P$ with itself.

Let $\mu$ be a measure generating and ergodic distribution $P_{x}$
at $x$. Note that there is a bijective correspondence between pairs
$(\sigma,\tau)\in\mathcal{M}_{1}^{\sqr}\times\mathcal{M}_{1}^{\sqr}$
and product measures $\sigma\times\tau\in\mathcal{M}_{2}^{\sqr}$,
and that the set of product measures is closed in $\mathcal{M}^{\sqr}\times\mathcal{M}^{\sqr}$.
Therefore the accumulation points of the sceneries of $\mu\times\mu$
are product measures, and if $\mu\times\mu$ generates a scenery flow
it is supported on product measures. Furthermore, if $\mu\times\mu$
generates a scenery flow $P_{x,y}$ at $(x,y)$ one may verify that,
making the identification between product measures and pairs, $P_{x,y}$
is a joining of the scenery flows $P_{x},P_{y}$ generated by $\mu$
at $x$ and $y$. Note that by Proposition \ref{pro:multidim-prediction-measures},
if $\mu$ is $T_{a}$-invariant then $\mu\times\mu$ indeed generates
sceneries a.e. 

Let $P$ be an $S^{\sqr}$-invariant and ergodic distribution with
an eigenvalue $\alpha$ and corresponding eigenfunction $\varphi$.
Then the function $p_{\alpha}:\mathcal{M}^{\sqr}\times\mathcal{M}^{\sqr}\rightarrow\mathbb{C}$,
defined by \[
p_{\alpha}(\sigma,\tau)=\frac{\varphi(\sigma)}{\varphi(\tau)}\]
is a.e. invariant on any ergodic $P$-joining, since\[
p_{\alpha}(S_{t}^{\sqr}\sigma,S_{t}^{\sqr}\tau)=\frac{\varphi(S_{t}^{\sqr}\sigma)}{\varphi(S_{t}^{\sqr}\tau)}=\frac{e(\alpha t)\varphi(\sigma)}{e(\alpha t)\varphi(\tau)}=p_{\alpha}(\sigma,\tau)\]
Therefore if $R$ is a $P$-joining then $p_{\alpha}$ is constant
on a.e. ergodic component of $R$, and if $R$ is ergodic then we
may define\begin{eqnarray*}
p_{\alpha}(R) & = & R\mbox{-a.s. value of }p_{\alpha}(\cdot,\cdot)\end{eqnarray*}

Let $\mu$ generate some distribution $P$ at a.e. point and assume
that $\mu\times\mu$ generates an ergodic $P$-joining $P_{x,y}$
at $\mu\times\mu$-a.e. point. Fix a $\mu$-typical $x_{0}$ so that
$P_{x_{0},y}$ is defined for $\mu$-a.e. $y$ and for such $y$ let
\[
p_{\alpha}(\mu,x_{0},y)=p_{\alpha}(P_{x_{0},y})\]

\begin{defn}
\label{def:phase-measure}Let $\mu$ be a measure which a.e. generates
$P$, and such that $\mu\times\mu$ generates an ergodic distribution
at a.e. point. Let $\alpha\in\Sigma_{P}$. For a $\mu$-typical point
$x_{0}$ the \emph{phase measure }$\theta_{\alpha}=\theta_{\alpha}(\mu,x_{0})$
is the push-forward of $\mu$ under the map $y\mapsto p_{\alpha}(\mu,x_{0},y)$.
\end{defn}
Under the further assumption that $\mu\times\mu\times\mu$ generates
an ergodic distribution at a.e. point, the dependence of the phase
measure on $x_{0}$ is very mild. Indeed, if we choose another point
$x_{1}$ then, by considering the threefold joining $Q$ generated
by $\mu\times\mu\times\mu$ at $(x_{0},y,x_{1})$ we find that, writing
$(\sigma_{1},\sigma_{2},\sigma_{3})$ for a $Q$-typical element,
\begin{eqnarray*}
p_{\alpha}(\mu,x_{1},y) & = & \frac{\varphi(\sigma_{3})}{\varphi(\sigma_{2})}\\
 & = & \frac{\varphi(\sigma_{3})}{\varphi(\sigma_{1})}\cdot\frac{\varphi(\sigma_{1})}{\varphi(\sigma_{2})}\\
 & = & c\cdot p_{\alpha}(\mu,x_{0},y)\end{eqnarray*}
where $c=p_{\alpha}(\mu,x_{1},x_{0})$ does not depend on $y$. Thus
$p_{\alpha}(\mu,x_{0},\cdot)$ depends on $x_{0}$ only up to a rotation,
and the measures $\theta_{\alpha}(\mu,x_{0})$ and $\theta_{\alpha}(\mu,x_{1})$
are rotations of one another. Since we shall be interested in properties
which are independent of rotation, we often suppress the dependence
on $x_{0}$ and write $\theta_{\alpha}(\mu)$. 
\begin{thm}
\label{thm:phase-distribution}Let $\mu$ be a $T_{a}$-invariant
measure of intermediate entropy whose ergodic components a.s. generate
the same distribution $P$. Let $\alpha\in\Sigma_{P}$ and $\theta_{\alpha}=\theta_{\alpha}(\mu)$
. Then
\begin{enumerate}
\item If $\alpha\in\Sigma_{P}\setminus\frac{1}{\log a}\mathbb{Q}$ then
$\theta_{\alpha}$ is Lebesgue measure.
\item If $\alpha\in\Sigma_{P}\cap\frac{1}{\log a}\mathbb{Q}$ then $\theta_{\alpha}$
is singular with respect to Lebesgue measure. 
\item If $\alpha\in\Sigma_{P}\cap\frac{1}{\log a}\mathbb{Z}$ and $\mu$
is ergodic then $\theta_{\alpha}$ consists of a single atom.
\end{enumerate}
\end{thm}
We do not know how large the phase measure can be for non-ergodic
$\mu$. With our methods we can go a little further and show that
$\ldim\theta_{\alpha}(\mu)\leq1-\dim\mu$, where $\ldim$ is the lower
Hausdorff dimension of $\mu$. On the other hand, starting with a
$T_{a}$-ergodic measure $\mu$ with $\alpha=\frac{n}{\log a}$ in
the spectrum of the scenery flow, the measure $\nu=\sum_{b=1}^{\infty}2^{-b}T_{b}\mu$
is again $T_{a}$-invariant and its phase measure consists of atoms
at $e(\log b/\log a)$, $b\in\mathbb{N}$. We suspect that the phase
measure is always atomic or at least of dimension zero, but we have
not resolved this. 

We next examine how the phase distribution changes when a smooth map
is applied to a measure. First, suppose that $\mu$ is a measure satisfying
the conditions in Definition \ref{def:phase-measure} and the discussion
following it. In particular for $\mu\times\mu$-a.e. $(x,y)$ the
scenery $(\mu_{x,t}\times\mu_{y,t})_{t\geq0}$ equidistributes for
some $P$-joining $P_{x,y}$. Now let $s\in\mathbb{R}$ and consider
the family $(\mu_{x,t}\times\mu_{y,t+s})_{t\geq0}$, in which we have
shifted the second component by $S_{s}^{\sqr}$. This family equidistributes
for the $P$-joining $Q=(\id\times S_{s}^{\sqr})P_{x,y}$ obtained
as the push-forward of $P_{x,y}$ through the map $(\sigma,\tau)\mapsto(\sigma,S_{s}^{\sqr}\tau)$.%
\footnote{This requires a short argument since $S^{\sqr}$ is not continuous,
but we omit it. %
} For $\alpha\in\Sigma_{P}$ and corresponding eigenfunction $\varphi$,
let $(\sigma,\tau)$ be a $P_{x,y}$-typical pair such that $(\sigma,S_{s}^{\sqr}\tau)$
is $Q$-typical. Then\begin{eqnarray*}
p_{\alpha}(Q) & = & p_{\alpha}(\sigma,S_{s}^{\sqr}\tau)\\
 & = & \frac{\varphi(\sigma)}{e(\alpha s)\varphi(\tau)}\\
 & = & e(-\alpha s)\cdot p_{\alpha}(\sigma,\tau)\\
 & = & e(-\alpha s)\cdot p_{\alpha}(P_{x,y})\end{eqnarray*}
Together with Lemma \ref{lem:diffeo-shifts-sceneries}, this leads
to the following result:
\begin{prop}
\label{pro:effect-of-diffeo-on-phase}Let $\mu$ be a $T_{a}$-invariant
measure of intermediate dimension which a.e. generates a distribution
$P$, let $\alpha\in\Sigma_{P}$ and let $f\in\diff^{1}(\mathbb{R})$.
Then $\theta_{\alpha}(f\mu)$ is well defined, and, fixing a $\mu$-typical
$x_{0}$, is given up to rotation by \[
\theta_{\alpha}(f\mu)=\int\delta_{e(-\alpha\log f'(y))\cdot p_{\alpha}(P_{x_{0},y})}\, d\mu(y)\]

\end{prop}
The discussion above proves the proposition when $f$ preserves orientation.
An obvious modification of the statement and proof is needed when
$f$ is orientation-reversing. See the remark after Lemma \ref{lem:diffeo-shifts-sceneries}.

\subsection{\label{sub:Proof-of-the-main-results}Proof of the main results}
\begin{proof}
[Proof (of Theorem \ref{thm:self-rigidity})] Let $\mu,\nu$ be $T_{a}$-ergodic
measures of intermediate dimension. Suppose $f\in\diff^{1}(\mathbb{R})$
and $f\mu|_{E}\sim\nu|_{E}$ for some set $E$ with $\nu(E)>0$. Then
$\theta_{\alpha}(f\mu|_{E})$ and $\theta_{\alpha}(\nu|_{E})$ are
equivalent for every $\alpha\in\Sigma_{P}$. Choosing $n\in\mathbb{N}$
and $\alpha=\frac{n}{\log a}\in\Sigma_{P}$, as we may by Theorem
\ref{thm:identification-of-spectrum}, it follows from Theorem \ref{thm:phase-distribution}
that $\theta_{\alpha}(\nu)$ is a point mass. Therefore $\theta_{\alpha}(f\mu|_{E})$
is a point mass, and by Proposition \ref{pro:effect-of-diffeo-on-phase}
this implies that $e(\alpha\cdot\log f'(\cdot))$ is $\mu$-a.s. constant
on $f^{-1}E$, giving the result. 

In the case $\nu=\mu$ we have $f\mu|_{E}\ll\mu|_{E}$ so $\theta_{\alpha}(f\mu|_{E})\ll\theta_{\alpha}(\mu)$,
and since both consist of a single atom we have equality. Hence $e(\alpha\cdot\log f'(\cdot))=1$
at $\mu$-a.e. point of $f^{-1}E$, so $f'|_{f^{-1}E}$ is $\mu$-a.e.
an integer power of $a^{1/n}$.
\end{proof}
Before proving the next theorems we require one more techinical result:
\begin{prop}
\label{cor:singularity-of-perturbed-phase-measure}Let $\mu$ be a
$T_{a}$-invariant measure of intermediate entropy generating $P$,
and $\alpha\in\Sigma_{P}\cap\frac{1}{\log a}\mathbb{Z}$. Suppose
one of the following holds:
\begin{enumerate}
\item $\mu$ is ergodic and $f\in\diff^{2}(\mathbb{R})$,
\item $f$ is affine.
\end{enumerate}
Then $\theta_{\alpha}(f\mu)$ is singular with respect to Lebesgue.\end{prop}
\begin{proof}
In the first case $\theta_{\lambda}(\mu,x_{0})$ consists of a single
atom (Theorem \ref{thm:phase-distribution}), and $p_{\lambda}(x_{0},\cdot)$
is independent of $y$. Then by Proposition \ref{pro:effect-of-diffeo-on-phase},
up to rotation $\theta_{\alpha}(f\mu)$ is the image of $\mu$ under
$g:x\mapsto e(f'(x))$. Since $f\in C^{2}$ we have $f'\in C^{1}$
and in particular $f'$, and hence $g$, is Lipschitz. Since $\mu$
has intermediate dimension it is singular with respect to Lebesgue,
so this is also true of $g\mu$, as desired.

In the second case $f(x)=ux+v$ and we need only consider the case
$u\neq0$. Since $f'(x)$ does not depend on $x$, we find by Proposition
\ref{pro:effect-of-diffeo-on-phase} that $\theta_{\alpha}(f\mu)$
is a rotation of $\theta_{\alpha}(\mu)$ by $e(\log u)$, and so,
since $\theta_{\lambda}$ is singular by Theorem \ref{thm:phase-distribution},
so is $\theta_{\alpha}(f\mu)$.
\end{proof}

\begin{proof}
[Proof (of Theorem \ref{thm:smooth-rigidity} and variants)] Let $\mu$
be a $T_{a}$-invariant measure and $\nu$ a $T_{b}$-invariant measure,
both of intermediate entropy. Suppose that $f\in\diff^{2}(\mathbb{R})$
and that $(f\mu)|_{E}\sim\nu|_{E}$ for some $E$ with $\nu(E)>0$.
Then $\theta_{\lambda}(f\mu|_{E})$ is equivalent to $\theta_{\lambda}(f\nu|_{E})$
and it suffices to show that this is impossible.

Assume that $\mu$ is ergodic, let $P$ denote the distribution generated
by $\mu$, and choose $\alpha\in\Sigma_{P}\cap\frac{1}{\log a}\mathbb{Z}$,
which is possible by Theorem \ref{thm:existence-of-spectrum}. Then
by Proposition \ref{cor:singularity-of-perturbed-phase-measure},
$\theta_{\alpha}(f\mu)$ is singular with respect to Lebesgue, while
by Theorem \ref{thm:phase-distribution} $\theta_{\alpha}(\nu)$ is
absolutely continuous. Thus the two are not equivalent. Note that
for this argument we did not require ergodicity of $\nu$.

Assume instead that $f$ is affine (but $\mu,\nu$ need not be ergodic).
We first disintegrate $\mu$ according to the partition of $[0,1]$
determined by the level sets of $x\mapsto P_{x}$, where $P_{x}$
is the distribution generated by $\mu$ at $x$. Since $P_{x}$ depends
only on $\mu^{(x)}$, this is a coarsening of the ergodic decomposition.
Decompose $\nu$ similarly. By Lemma \ref{lem:diffeo-shifts-sceneries}
$f$ respects these partitions, so it suffices to prove the result
for the corresponding conditional measures, which in the case of $\mu$
are $T_{a}$-invariant, and $T_{b}$-invariantin the case of $\nu$.
Hence we may assume from the start that $\mu,\nu$ generate a single
distribution $P$ a.e. The result now follows as above from the second
part of Corollary \ref{cor:singularity-of-perturbed-phase-measure}.

The case of $C^{1}$-maps when $\mu,\nu$ satisfy some spectral assumptions
was sketched after Theorem \ref{thm:existence-of-spectrum}.
\end{proof}

\begin{proof}
{Proof (of Corollary \ref{cor:conjugates-rigidity})} Suppose $\varphi f\varphi^{-1}=T_{a}$
and $\psi g\psi^{-1}=T_{b}$ for $\varphi,\psi\in\diff^{2}(\mathbb{R}/\mathbb{Z})$.
Let $\mu$ be a common measure of positive dimension. The measures
$\varphi\mu$,$\psi\mu$ are invariant, respectively, for $T_{a}$
and $T_{b}$, and the dimension hypothesis implies that they have
no ergodic component of entropy zero. Now $\psi\varphi^{-1}(\varphi\mu)=\varphi\mu$
and $\psi\varphi^{-1}\in\diff^{2}(\mathbb{R}/\mathbb{Z})$, so Theorem
\ref{thm:smooth-rigidity} implies that $\varphi\mu$ is Lebesgue.
Therefore $\psi\varphi^{-1}(\varphi\mu)$ is a $T_{b}$-invariant
measure equivalent to Lebesgue, so it must be Lebesgue measure as
well, and so $\psi\varphi^{-1}$ preserves Lebesgue measure. The only
diffeomorphism on $\mathbb{R}/\mathbb{Z}$ with this property is a
rotation.
\end{proof}

\section{\label{sec:The-scenery-process}Construction and properties of the
scenery flow}

Throughout this section we fix an integer $b>1$ and a non-atomic
probability measure $\mu$ on $[0,1]$ which is invariant under $T_{b}$.
We write $[u;v)=[u,v)\cap\mathbb{Z}$, and similarly $[u;v]$ etc.
Our convention is $\mathbb{N}=\{1,2,3\ldots\}$.

\subsection{\label{sub:extended-scenery-flow} The extended scenery flow }

For the moment fix the dimension $d=1$ and consider measures on $\mathbb{R}$.
We use $*$ to denote the operation of normalizing a measure on $[-1,1]^{d}$
, that is, if $\tau$ is a Radon measure on $\mathbb{R}^{d}$ and
$\tau([-1,1])>0$, then\[
\tau^{*}=\frac{1}{\tau([-1,1])}\tau\]
Thus $\tau^{\sqr}=(\tau^{*})|_{[-1,1]}$. Let $\mathcal{M}^{*}\subseteq\mathcal{M}$
denote the set of measures giving unit mass to $[-1,1]$. Write $S_{t}^{*}:\mathcal{M}^{*}\rightarrow\mathcal{M}^{*}$
for the partially defined map \[
S_{t}^{*}\mu=(S_{t}\mu)^{*}\]
Thus $S^{*}=(S_{t}^{*})_{t\in\mathbb{R}}$ is a measurable flow on
the Borel subset of measures $\mu\in\mathcal{M}^{*}$ with $0$ in
their support.

While working with $S^{*}$ is more natural than $S^{\sqr}$, we used
the latter in the definition of the scenery flow because $\mathcal{P}(\mathcal{M}^{*})$
does not carry a nice topology with which to define equidistribution
of $S^{*}$-orbits. However there is a simple way to move between
invariant distributions of the two flows. First, one may verify that
$\tau\mapsto\tau^{\sqr}$ is a factor map from the measurable flow
$(\mathcal{M}^{*},S^{*})$ to the semi-flow $(\mathcal{M}^{\sqr},S^{\sqr})$,
i.e. $S_{t}^{\sqr}(\mu^{\sqr})=(S_{t}^{*}\mu)^{\sqr}$, and so an
$S^{*}$-invariant distribution $Q$ is pushed via $\mu\mapsto\mu^{\sqr}$
to an $S^{\sqr}$-invariant distribution $P=Q^{\sqr}$ called the
restricted version of $Q$. Conversely, if $P$ is a an $S^{\sqr}$-invariant
distribution then there is a unique $S^{*}$-invariant distribution
$Q$ on $\mathcal{M}^{*}$, called the \emph{extended version }of
$P$, satisfying $Q^{\sqr}=P$. The extended version may be obtained
as the inverse limit of the diagram \[
\ldots\xrightarrow{S_{1}^{\sqr}}\mathcal{M}^{\sqr}\xrightarrow{S_{1}^{\sqr}}\mathcal{M}^{\sqr}\xrightarrow{S_{1}^{\sqr}}\mathcal{M}^{\sqr}\]
(so dynamically $Q$ is the natural extension of $P$). Indeed, starting
from a left-infinite sequence $(\ldots,\mu_{-2},\mu_{-1},\mu_{0})$
with $S_{1}^{\sqr}\mu_{i+1}=\mu_{i}$, there is a unique measure $\mu_{\infty}\in\mathcal{M}^{*}$
such that $\mu_{\infty}|_{[-b^{n},b^{n}]}=S_{n}^{*}\mu_{-n}$, and
the induced distribution $Q$ on these measures is seen to be $S^{*}$-invariant
and satisfy $Q^{\sqr}=P$.

We shall usually not make the distinction between the extended and
restricted versions of these flows. We note for later use that they
have the same pure point spectrum.

\subsection{\label{sub:prediction-measures-as-scenery}The scenery flow of a
$T_{a}$-invariant measure}

In this section we construct a flow associated to a $T_{b}$-invariant
measure, and study its properties. Let $\Omega=\Omega_{b}=\{0,\ldots,b-1\}^{\mathbb{Z}}$,
and denote the shift map on $\Omega$ by $T$, i.e. $(T\omega)_{i}=\omega_{i+1}$.
We write $\omega_{I}$ for the subsequence $(\omega_{i})_{i\in I}$. 

Let $\xi_{k}:\Omega\rightarrow[0,b^{-k}]$ denote the map\[
\xi_{k}(\omega)=\sum_{i=k+1}^{\infty}b^{-i}\omega_{i}\]
In particular we write\[
\xi=\xi_{0}\]
which is the base-$b$ coding map taking $\omega\in\Omega$ to the
point $x\in[0,1]$ whose base-$b$ expansion is $0.\omega_{1}\omega_{2}\ldots$
(note that this map is everywhere uncountable-to-one since it discards
the non-positive coordinates of $\omega$).

Every non-atomic $T_{b}$-invariant measure $\mu$ lifts to a unique
$T$-invariant measure $\widetilde{\mu}$ on $\Omega$ such that $\xi\widetilde{\mu}=\mu$;
the system $(\widetilde{\mu},T)$ is a realization of the natural
extension of $(\mu,T_{b})$. For $x\in[0,1]$ we denote by $\mu^{(x)}$
the ergodic component of $x$ in $\mu$, and write $\widetilde{\mu}^{(x)}$
for the unique ergodic component of $\widetilde{\mu}$ which maps
under $\xi_{0}$ to $\mu^{(x)}$. We also write $\widetilde{\mu}^{(\omega)}$
for the ergodic component of $\omega\in\Omega$.

Recall that a measure on $\mathbb{R}$ has exact dimension $\alpha$
if its local dimension exists a.e. and is a.e. equal to $\alpha$.

Below we construct a map $\pi:\Omega\rightarrow\mathcal{M}$, usually
denoted $\omega\mapsto\mu_{\omega}$, defined a.e. for every $T$-invariant
measure on $\Omega$, and satisfying the following properties.
\begin{prop}
\label{pro:prediction-measures-and-their-properties}For every $T_{b}$-invariant
measure $\mu$, the following hold:
\begin{enumerate}
\item \label{enu:prediction-measure-superposition} $\mu^{(x)}$ can be
represented as\[
\mu^{(x)}=\int(U_{-\xi(\omega)}\mu_{\omega})|_{[0,1]^{d}}\, d\widetilde{\mu}^{(x)}(\omega)\]
and in particular\[
\mu=\int(U_{-\xi(\omega)}\mu_{\omega})|_{[0,1]^{d}}\, d\widetilde{\mu}(\omega)\]

\item \label{enu:scenery-flow-comes-from-prediction-meas}The map $\pi^{*}:\omega\mapsto\mu_{\omega}^{*}$
intertwines the actions of $T$ and $S_{\log b}^{*}$, i.e.\[
S_{\log b}^{*}\mu_{\omega}^{*}=\mu_{T\omega}^{*}\]
In particular, the distribution \[
\widetilde{P}_{x}=\pi^{*}\widetilde{\mu}^{(x)}\]
is $S_{\log b}^{*}$-invariant, and the map\[
\pi^{*}:(\Omega,\widetilde{\mu}^{(x)},T)\rightarrow(\mathcal{M}^{*},\widetilde{P}_{x},S_{\log b}^{*})\]
is a factor map of discrete-time systems.
\item \label{enu:dimension-of-prediction-measures}$\mu_{\omega}$ has exact
dimension $h(\widetilde{\mu}^{(\omega)})/\log b$.
\end{enumerate}
\end{prop}
We now begin the construction of $\mu_{\omega}$. Let $\mu$ be a
$T_{b}$-invariant measure. Given a left-infinite sequence $\omega_{(-\infty;k]}\in\{0,\ldots,b-1\}^{(-\infty;k]}$
let $\widetilde{\mu}(\cdot\,|\,\omega_{(-\infty;k]})$ be the probability
measure obtained by conditioning $\widetilde{\mu}$ on the set $\{\eta\in\Omega\,:\,\eta_{(-\infty;k]}=\omega_{(-\infty;k]}\}$.
Note that this set can be identified, using $\xi_{k}$, with $[0,b^{-k}]$.
Define the measure $\mu(\cdot|\omega_{(-\infty;k]})$ on $[0,b^{-k}]$
by pushing these conditional measure forward through $\xi_{k}$, i.e.
\[
\mu(\cdot\;|\;\omega_{(-\infty;k]})=\widetilde{\mu}(\xi_{k}^{-1}(\cdot)\;|\;\omega_{(-\infty;k]})\]
Note that this definition depends only on the ergodic component $\widetilde{\mu}^{(\omega)}$,
rather than the pair $(\widetilde{\mu},\omega)$.

For any measurable $A\subseteq[0,b^{-k}]$ we have the relation \[
\mu(A\;|\;\omega_{(-\infty;k]})=c_{\omega}^{k}\cdot\mu(A+b^{-k}\omega_{k}\;|\;\omega_{(-\infty;k-1]})\]
where $c_{\omega}^{k}$ is a normalizing constant chosen so that equality
holds for $A=[0,b^{-k}]$. More generally, for $k<m$ and $A\subseteq[0,b^{-m+1}]$,
we have\begin{equation}
\mu(A\;|\;\omega_{(-\infty;m]})=c_{\omega}^{k,m}\cdot\mu(A+\sum_{k<i\leq m}b^{-i}\omega_{i}\;|\;\omega_{(-\infty;k]})\label{eq:consistency}\end{equation}
where $c_{\omega}^{k,m}=c_{\omega}^{k+1}\cdot\ldots\cdot c_{\omega}^{m}$. 

It follows that the sequence of measures $\mu_{\omega,k}\in\mathcal{M}$,
defined for $k\leq0$ by \begin{equation}
\mu_{\omega,k}(A)=c_{\omega}^{k,0}\cdot\mu(A+\sum_{i=k+1}^{\infty}b^{-i}\omega_{i}\;|\;\omega_{(-\infty,k]})\label{eq:k-th-prediction-measure}\end{equation}
agree as $k\rightarrow-\infty$ on the increasing sequence of intervals\[
[-\xi_{k}(\omega),\xi_{k}(\omega)+b^{-k}]=[-\sum_{i=k+1}^{\infty}b^{-i}\omega_{i},\sum_{i=k+1}^{\infty}b^{-i}\omega_{i}+b^{-k}]\]
and vanish outside of them. Therefore, as $k\rightarrow-\infty$ the
measures $\mu_{\omega,k}$ converge to a Radon measure which we denote
\[
\mu_{\omega}=\lim_{k\rightarrow-\infty}\mu_{\omega,k}\]

Let us now verify the properties stated in Proposition \ref{pro:prediction-measures-and-their-properties}.
First, from equation \eqref{eq:k-th-prediction-measure} we see that
for $k<0$,\[
\mu_{\omega,k}(A-\xi_{0}(\omega))=c_{\omega}^{k,0}\cdot\mu(A+\sum_{i=k+1}^{0}b^{-i}\omega_{i}\;|\;\omega_{(-\infty,k]})\]
which, from equation \eqref{eq:consistency}, implies \[
\mu_{\omega}(A-\xi_{0}(\omega))=\mu(A\,|\,\omega_{(-\infty,0]})\]
Integrating over $\omega\sim\widetilde{\mu}^{(x)}$ or over $\omega\sim\widetilde{\mu}$
gives part \eqref{enu:prediction-measure-superposition} of Proposition
\ref{pro:prediction-measures-and-their-properties}.

Note that $0\in\supp\mu_{\omega}$ and the relation \eqref{eq:consistency}
implies \begin{equation}
\mu_{T\omega}(A)=c_{\omega}\mu_{\omega}(\frac{1}{b}A)\label{eq:consistency-2}\end{equation}
for some constant $c_{\omega}$ independent of $A$. Thus by \eqref{eq:consistency-2},
the map $\omega\mapsto\mu_{\omega}^{*}$ is defined $\widetilde{\mu}$-a.e.
and intertwines the shift $T$ and the scaling map $S_{\log b}^{*}$,
i.e. \[
S_{\log b}^{*}\mu_{\omega}^{*}=\mu_{T\omega}^{*}\]
This establishes part \eqref{enu:scenery-flow-comes-from-prediction-meas}
of Proposition \ref{pro:prediction-measures-and-their-properties}
(the later statements in that part follow from the first).

Let \[
[\omega_{1}\ldots\omega_{n}]=\{\eta\in\Omega\,:\,\eta_{1}\ldots\eta_{n}=\omega_{1}\ldots\omega_{n}\}\]
denote the cylinder set corresponding to a finite sequence $\omega_{1}\ldots\omega_{n}$.
A variant of the Shannon-McMillan-Breiman theorem states that for
$\widetilde{\mu}$-a.e. $\omega$, \[
\lim_{n\rightarrow\infty}\frac{1}{n}\log\widetilde{\mu}([\omega_{1},\ldots,\omega_{n}]\,|\,\omega_{(-\infty;0]})=h(\widetilde{\mu}^{(\omega)})\]
As a consequence for $\widetilde{\mu}$-a.e. $\omega$ the measure
$\tau=\mu(\cdot|\omega_{(-\infty,0]})$ has exact dimension $h(\widetilde{\mu}^{(\omega)})/\log b$,
i.e. \[
\lim_{r\searrow0}\frac{\log\tau(B_{r}(x))}{\log r}=\frac{h(\widetilde{\mu}^{(\omega)})}{\log b}\qquad\mbox{ at }\tau\mbox{-a.e. point }x\]
(what is obvious is that this limit holds when, instead of $B_{r}(x)$,
we consider the mass of $b$-adic intervals containing $x$, since
these correspond to cylinder sets; the version above follows using
e.g. \cite[Theorem 15.3]{Pesin97}). The same argument also holds
for $\mu(\cdot|\omega_{(-\infty,k]})$ and for any $k<0$, hence for
$\mu_{\omega,k}$, and gives the result for $\mu_{\omega}$. 

As a special case, we remark that when $h(\widetilde{\mu})=0$ all
the conditional measures $\mu(\cdot|\omega_{(-\infty,k]})$ consist
of a single atom, and consequently, $\mu_{\omega}=\delta_{0}$. Likewise,
when $\widetilde{\mu}$ is $\lambda^{*}$ it is easy to verify that
$\mu_{\omega}=\lambda^{*}$ for $\widetilde{\mu}$-a.e. $\omega$.
In these cases the flows $P_{x}$ are trivial, consisting of point
masses at the $S^{*}$-fixed points $\delta_{0}$ or $\lambda^{*}$.

\subsection{\label{sub:Convergence-of-sceneries}Convergence of the scenery to
the scenery flow}

We now turn to the sceneries of $T_{b}$-invariant measures and their
relation to the flow constructed above. We continue to work in dimension
$d=1$ and with the notation of the previous section.
\begin{prop}
\label{pro:convergence-of sceneries}Let $\mu$ be a $T_{b}$-invariant
measure and let $\widetilde{\mu},\widetilde{P}_{x}$ be as in Proposition
\ref{pro:prediction-measures-and-their-properties}. Then at $\mu$-a.e.
$x$, $\mu$ generates the $S^{*}$-invariant distribution\[
P_{x}=\int_{0}^{1}S_{t\log b}^{*}\widetilde{P}_{x}\, dt\]
In particular, the scenery flow generated by $\mu$ at $x$ depends
only on the ergodic component of $x$ and arises as a factor of the
$\log b$-suspension of $(\Omega,\widetilde{\mu}^{(x)},T)$.
\end{prop}
Note that, since by Proposition \ref{pro:identification of generated-distributions-1}
$\widetilde{P}_{x}$ is $S_{\log b}^{*}$ invariant and ergodic, it
follows immediately that $P_{x}$, as defined above, is $S^{*}$-invariant
and ergodic and that $(\mathcal{M}^{*},\widetilde{P}_{x},S_{\log b}^{*})$
is a factor of the $\log b$-suspension of the discrete time system
$(\mathcal{M}^{*},\widetilde{P}_{x},S_{\log b}^{*})$. Since $(\mathcal{M}^{*},\widetilde{P}_{x},S_{\log b}^{*})$
is a factor of $(\Omega,\widetilde{\mu}^{(x)},T)$ by Proposition
\ref{pro:prediction-measures-and-their-properties}, the statement
in the second part of the proposition above is automatic, and it remains
only to prove the first, i.e. that $\mu$ generates $P_{x}$ at $x$.

To this end we introduce another sequence of measures $\mu'_{\omega,k}$
on $[0,b^{-k}]$ by \[
\mu'_{\omega,k}(A)=c_{\omega}^{k,0}\cdot\widetilde{\mu}(\xi_{k}^{-1}(A+\sum_{i=k+1}^{\infty}b^{-i}\omega_{i}))\]
This is the same as the definition of $\mu_{\omega,k}$ except we
have not conditioned on $\omega_{(-\infty,k]}$
\begin{prop}
\label{pro:sceneries-asymp-to-prediction-measures}$\lim_{k\rightarrow-\infty}\mu'_{\omega,k}=\mu_{\omega}$
weak-{*} on any compact set in $\mathbb{R}$.\end{prop}
\begin{proof}
Let $\widetilde{\mu}(\cdot|\omega_{[k+1;\infty)})$ denote the conditional
measure $\widetilde{\mu}$ on sequences $\eta\in\Omega$ given the
{}``future'' $\eta_{[k+1;\infty)}=\omega_{[k+1;\infty)}$. Then
\begin{eqnarray*}
\mu'_{\omega,k} & = & \int\mu_{\eta,k}\, d\widetilde{\mu}(\eta|\omega_{[k+1;\infty)})\\
 & = & \int(\mu_{\eta,k}-\mu_{\eta})\, d\widetilde{\mu}(\eta|\omega_{[k+1;\infty)})\;+\;\int\mu_{\eta}\, d\widetilde{\mu}(\eta|\omega_{[k+1;\infty)})\end{eqnarray*}
The second term in the last expression is a measure-valued martingale
in the variable $\omega$ with respect to the filtration $\mathcal{F}_{k}\subseteq\mathcal{F}_{k-1},\ldots$,
where $\mathcal{F}_{k}$ is the $\sigma$-algebra generated by coordinates
$k+1,k+2,\ldots$. Since these algebras generate the Borel algebra
on $\Omega$, the term on the right converges $\widetilde{\mu}$-a.e.
to $\mu_{\omega}$ as $k\rightarrow-\infty$.

In order to deal with the first term on the right, note that if we
integrate against any compactly supported function $f$ on $\mathbb{R}$,
for $k>k_{0}(\omega)$ the measures $\mu_{\eta,k}$ and $\mu_{\eta,-\infty}$
will agree on the support of $f$, and so the integral will vanish.
Hence the first term also converges to $0$ weak-{*} on any compact
set. \end{proof}
\begin{cor}
$\lim_{k\rightarrow-\infty}(\mu'_{\omega,k})^{\sqr}=(\mu_{\omega})^{\sqr}$
in the weak-{*} sense for $\widetilde{\mu}$-a.e. $\omega$.
\end{cor}
For the next step we rely on a classical ergodic theorem due to Maker: 
\begin{thm}
[Maker, \cite{Maker40}]\label{thm:Maker} Let $(\Omega,\nu,T)$ be
a measure preserving system. Let $F_{n}$ be measurable functions
with $\sup_{n}|F_{n}|\in L^{1}$ and suppose that $F_{n}\rightarrow F$
a.e.. Then \[
\frac{1}{N}\sum_{n=1}^{N}T^{n}F_{n}\rightarrow\mathbb{E}(F|\mathcal{E})\]
 a.e., where $\mathcal{E}$ is the $\sigma$-algebra of $T$-invariant
sets.\end{thm}
\begin{prop}
\label{pro:convergence-of-discrete-sceneries}For $\mu$-a.e. $x$
\[
\lim_{N\rightarrow\infty}\frac{1}{N}\sum_{n=1}^{N}\delta_{\mu_{x,n\log b}}\rightarrow\widetilde{P}_{x}\]
\end{prop}
\begin{proof}
Define $F_{k}(\omega)=\delta_{(\mu'_{\omega,-k})^{\sqr}}$ and $F_{k}(\omega)=\delta_{(\mu_{\omega})^{\sqr}}$.
From the corollary we know that $F_{k}\rightarrow F$ a.s. and these
distribution-valued functions are uniformly bounded in the space of
distributions on $\mathcal{P}(\mathcal{M}^{\sqr})$. Thus by Maker's
theorem,\[
\frac{1}{N}\sum_{n=1}^{N}F_{k}(T^{k}\omega)\rightarrow\mathbb{E}(F\,|\,\mathcal{E})\]
$\widetilde{\mu}$-a.e. One verifies from the definitions that\[
(\mu'_{T^{k}\omega,k})^{\sqr}=\mu_{x,k\log k}\]
for $x=\xi_{0}(\omega)$. The proposition follows, since $\mathbb{E}(\delta_{(\mu_{\omega})^{\sqr}}\,|\,\mathcal{E})=\widetilde{P}_{x}$.
\end{proof}
Finally, applying the operator $\int_{0}^{1}S_{t\log b}^{\sqr}\, dt$
to the limit in the proposition above, we arrive at the conclusion
of Proposition \ref{pro:prediction-measures-and-their-properties}
concerning continuous-time sceneries.

\subsection{\label{sub:The-multidimensional-case}The multidimensional case}

We now turn to the higher dimensional setting. Let $*$ denote the
normalization operation $\mu\mapsto\mu^{*}=\frac{1}{\mu([-1,1]^{d})}\mu$
and define the associated objects as in Section \ref{sub:extended-scenery-flow}. 

Let $\mu_{1},\ldots\mu_{d}$ by $T_{b}$-invariant measures and write
$\mu=\mu_{1}\times\ldots\times\mu_{d}$, which is a measure on $[0,1]^{d}$
invariant under the diagonal map $T_{b}(x)=(T_{b}x_{1},\ldots,T_{b}x_{d})$.
The measure $\widetilde{\mu}=\widetilde{\mu}_{1}\times\ldots\times\widetilde{\mu}_{d}$
on $\Omega^{d}\cong(\{0,\ldots,b-1\}^{d})^{\mathbb{Z}}$ is the natural
extension of $\mu$. We may define maps $\xi_{k}:\Omega^{d}\rightarrow[0,b^{-k}]^{d}$
by applying $\xi_{k}$ coordinatewise and define a map $\omega\mapsto\mu_{\omega}\in\mathcal{M}_{d}$
using the same procedure as in dimension $d=1$.

Let $\varphi_{i}$ denote projection to the $i$-th coordinate. Let
$\pi^{*}$ and $\pi_{i}^{*}$ denote the maps $\omega\mapsto\mu_{\omega}^{*}$
and $\eta\mapsto(\mu_{i})_{\eta}^{*}$ respectively.
\begin{prop}
\label{pro:multidim-prediction-measures}Let $\mu_{1},\ldots,\mu_{d}$
be $T_{b}$-invariant measures on $[0,1]$, let $\mu=\times_{i=1}^{d}\mu_{i}$,
and let $\omega\mapsto\mu_{\omega}$ be as above. Then:\end{prop}
\begin{enumerate}
\item \label{enu:convergence-for-products}The analogs of Propositions \ref{pro:prediction-measures-and-their-properties}
and \ref{pro:convergence-of sceneries} hold.
\item \label{enu:prediction-measure-for-products}$\mu_{\omega}=\times_{i=1}^{d}(\mu_{i})_{\omega^{i}}$.
\item \label{enu:diagram}If $\mu$ generates $\widetilde{P}_{x}$ at $x\in[0,1]^{d}$
and $\mu_{i}$ generates $\widetilde{P}_{i,y}$ at $y$, and we identify
product measures with $d$-tuples of measures, then the following
diagram factor maps commutes:\begin{eqnarray*}
([0,1]^{d},\widetilde{\mu}^{(x)},T) & \xrightarrow{\pi^{*}} & (\mathcal{M}^{d},\widetilde{P}_{x},S_{\log b}^{*})\\
\varphi_{i}\downarrow\qquad &  & \qquad\downarrow\varphi_{i}\\
([0,1],\widetilde{\mu}_{i}^{(x_{i})},T) & \xrightarrow{\pi_{i}^{*}} & (\mathcal{M},\widetilde{P}_{i,x_{i}},S_{\log b}^{*})\end{eqnarray*}
\end{enumerate}
\begin{proof}
The proof of \eqref{enu:convergence-for-products} is the same as
the 1-dimensional case. 

For \eqref{enu:prediction-measure-for-products}, note that the space
of product measures on $\mathbb{R}^{d}$ is closed and each $\mu_{x,t}$
is a product measure, so the scenery flow of a product measure is
supported on product measures. Hence by part \eqref{enu:convergence-for-products},
$\mu_{\omega}$ is a product measure. 

In order to see that $\mu_{\omega}=\times_{i=1}^{d}\mu_{i,\omega^{i}}$,
let $Q$ be the $S_{\log b}^{*}$-invariant distribution on pairs
of measures obtained by pushing forward $\widetilde{\mu}$ through
the map $\omega\mapsto((\mu_{1})_{\omega^{1}}^{*},\mu_{\omega}^{*})$,
and let $Q_{x}$ be the push-forward of $\widetilde{\mu}^{(x)}$ by
the same map, so that $Q=\int Q_{x}\, d\widetilde{\mu}(x)$. We wish
to show that $Q$-a.e. pair $(\tau,\nu_{1}\times\ldots\times\nu_{d})$
satisfies $\tau=\nu_{1}$, so we must show that for $\mu$-a.e. $x$
this holds for $Q_{x}$. To see this, consider for $\mu$-typical
$x\in[0,1]^{d}$ the sequence \[
(\tau_{n},\nu_{n})=((\mu_{1})_{x_{1},n\log b},\mu_{x,n\log b})\]
$n=1,2,3\ldots$, and repeat the proof of Proposition \ref{pro:convergence-of sceneries}
to conclude that $\frac{1}{N}\sum_{n=1}^{N}\delta_{(\tau_{n},\nu_{n})}\rightarrow Q_{x}$.
Since the relationship $\tau_{n}=\pi_{1}(\nu_{n})$ holds for all
$n$ and this is a closed condition, it holds also for the limiting
distribution $Q_{x}$. 

Finally, the commutativity of the diagram is a direct result of the
relationship $(\mu_{i})_{\omega^{i}}=\pi_{i}(\mu_{\omega})$.
\end{proof}

\subsection{\label{sub:Eigenvalues-and-ergodicity}Eigenvalues and ergodicity
of flows}

We briefly present some technical facts about flows and their spectrum
and ergodicity properties. For sake of economy we present the discussion
for an invariant distribution $P$ on $(\mathcal{M}^{*},S^{*})$. 

A function $f\in L^{2}(P)$ generally is defined only $P$-a.e. and
hence for typical $\nu$ it is defined at $S_{t}^{*}\nu$ for only
Lebesgue-a.e. $t$. The next lemma says that a function which behaves
like an eigenfunction at a.e. point along a.e. orbit may be modified
on a set of measure zero to become an eigenfunction.
\begin{lem}
\label{lem:canonical-eigenfunctions}Let $\varphi\in L^{2}(P)$ and
suppose that for every $t\in\mathbb{R}$ we have $S_{t}^{*}\varphi=e(\alpha t)\varphi$
$P$-a.e. Then there exists $\overline{\varphi}\in L^{2}(P)$ which,
for $P$-a.e. $\nu$, satisfies $S_{t}^{*}\overline{\varphi}(\nu)=e(\alpha t)\varphi(\nu)$
for every $t\in\mathbb{R}$, and $\varphi=\overline{\varphi}$ a.e.\end{lem}
\begin{proof}
Define $\overline{\varphi}(\nu)=\int_{0}^{1}e(-\alpha t)\varphi(S_{t}^{*}\nu)$
\end{proof}
Usually, the ergodic decomposition of a measure is defined only in
an a.e. sense. For the decomposition of $P$ with respect to $S_{t_{0}}^{*}$
we can give a more canonical description. We say that an $S_{t_{0}}^{*}$-invariant
distribution $Q$ is an ergodic component of $P$ (with respect to
$S_{t_{0}}^{*}$) if it is ergodic for $S_{t_{0}}^{*}$ and $\int_{0}^{1}S_{t_{0}\cdot t}^{*}Q\, dt=P$.
Note that if $P=\int Q_{\nu}\, dP(\nu)$ is an abstract ergodic decomposition
of $P$ with respect to $S_{t_{0}}^{*}$ then $P=\int(\int_{0}^{1}S_{t_{0}\cdot t}^{*}Q_{\nu}\, dt)\, dP(\nu)$,
and each of the inner integrals is $S^{*}$-invariant. Therefore,
ergodicity of $P$ implies that for $P$-a.e. $\nu$ the inner integral
is $P$, so $Q_{\nu}$ is an ergodic component. Hence ergodic components
exist.
\begin{lem}
\label{lem:canonical-ergodic-components}If $Q$ and $Q'$ are ergodic
components for $S_{t_{0}}^{*}$ then $S_{t_{0}\cdot r}^{*}Q'=Q$ for
some $r\in[0,1]$. In particular, the representations of $P$ as $\int_{0}^{1}S_{t_{0}\cdot t}^{*}Q\, dt$
does not depend (up to a translation modulo 1 of the parameter space)
on the ergodic component $Q$. \end{lem}
\begin{proof}
Since $\int_{0}^{1}S_{t_{0}\cdot t}^{*}Q\, dt$, $\int_{0}^{1}S_{t_{0}\cdot t}^{*}Q'\, dt$
are both ergodic decompositions of $P$, by  uniqueness of the ergodic
decomposition we see that for a.e. $t\in[0,1]$ there is an $s\in[0,1]$
with $S_{t_{0}\cdot t}^{*}Q=S_{t_{0}\cdot s}^{*}Q'$. Then for $r=s-t$
(or $r=1+s-t$ if $s<t$), we have $Q=S_{t_{0}(s-t)}^{*}Q'$. The
second statement is immediate from the first.\end{proof}
\begin{lem}
\label{lem:invariance-of-ergodic-components}Let $Q$ be an ergodic
component of $P$ with respect to $S_{t_{0}}^{*}$. Then either $Q=P$,
i.e. $S_{t_{0}}^{*}$ is ergodic, or there is a largest $n\in\mathbb{N}$
such that $Q$ is invariant under $S_{t_{0}/n}^{*}$.\end{lem}
\begin{proof}
Consider the map $q:\mathbb{R}/\mathbb{Z}\rightarrow\mathcal{P}(\mathcal{M})$
given by $q(t)=S_{t_{0}\cdot t}^{*}Q$, which is well defined since
$S_{t_{0}}^{*}Q=Q$. Let $\Lambda\subseteq\mathbb{R}/\mathbb{Z}$
denote the set of periods of $q$, that is $r\in\Lambda$ if $q(t+r)=q(t)$
for $t\in\mathbb{R}/\mathbb{Z}$, or equivalently, such that $q(r)=q(0)$.
Since $q$ is measurable with respect to Lebesgue measure, either
$\Lambda$ is discrete or $\Lambda=\mathbb{R}/\mathbb{Z}$. In the
latter case, since $P=\int_{\mathbb{R}/\mathbb{Z}}q(t)\, dt=q(0)=Q$.
In the former case $\Lambda$ has the form $\{\frac{k}{n}\,:\,0\leq k<n\}$
for some $n$, and this is the $n$ we are looking for.\end{proof}
\begin{lem}
\label{lem:ergodicity-vs-spectrum}For $t_{0}>0$ the following are
equivalent:
\begin{enumerate}
\item $\frac{1}{t_{0}}\in\Sigma_{P}$.
\item For some (equivalently every) ergodic component $Q$ of $S_{t_{0}}^{*}$,
the flow $(\mathcal{M}^{*},P,S^{*})$ is isomorphic to the $t_{0}$-suspension
of $(\mathcal{M}^{*},Q,S_{t_{0}}^{*})$.
\item $S_{t_{0}}^{*}$ is not ergodic, and its ergodic components are not
preserved under $S_{t_{0}/n}^{*}$ for any $n\in\mathbb{N}$.
\end{enumerate}
\end{lem}
\begin{proof}
[Proof (sketch)] (1)$\implies$(2): If $\varphi$ is an eigenfunction
for $\frac{1}{t_{0}}$ then one may verify that the ergodic components
of $S_{t_{0}}^{*}$ are precisely the conditional distributions of
$P$ on the level sets of $\varphi$. Fixing an ergodic component
$Q$ supported on a level set $\varphi^{-1}(z)$ define $r:\mathcal{M}^{*}\rightarrow[0,1)$
by $e(t_{0}r(\nu))=z$. Then $\nu\mapsto(S_{-r(\nu)}^{*}\nu,\varphi(\nu))$
is an isomorphism of $(\mathcal{M}^{*},P,S^{*})$ and the $t_{0}$-suspension
of $(\mathcal{M}^{*},Q,S_{t_{0}}^{*})$.

(2)$\implies$(3): Trivial since e.g. the subset of $\mathcal{M}^{*}$
corresponding to $\mathcal{M}^{*}\times[0,\frac{t_{0}}{2})$ in the
suspension is $S_{t_{0}}^{*}$-invariant, but not $S_{t_{0}/n}^{*}$-invariant
for any $1\neq n\in\mathbb{N}$.

(3)$\implies$(1): By the previous lemma we find that the action of
$S^{*}$ on the ergodic components for $S_{t_{0}}^{*}$ is isomorphic
to $[0,t_{0})$ with addition modulo $t_{0}$. Since $P$-a.e. point
belongs to a well defined ergodic component, this gives an eigenfunction
with eigenvalue $\frac{1}{t_{0}}$.
\end{proof}

\subsection{\label{sub:Existence-of-log-b-in-spectrum} The spectrum of $(\mathcal{M}^{*},P_{x},S^{*})$}

In this section we prove Theorem \ref{thm:identification-of-spectrum}.
Let $\mu$ be $T_{b}$-ergodic with entropy strictly between $0$
and $\log b$. Recall the construction and notation from Section \ref{sub:prediction-measures-as-scenery}:
specifically $(\Omega,\widetilde{\mu},T)$ is the natural extension
of $([0,1],\mu,T_{b})$, the image of $\widetilde{\mu}$ under $\omega\mapsto\mu_{\omega}$
is the $S_{\log b}^{*}$-ergodic distribution $\widetilde{P}$ (it
does not depend on $x$ or $\omega$, as it did in previous sections,
because $\mu$ and $\widetilde{\mu}$ are now ergodic), and $P=\int_{0}^{1}S_{t\log b}^{*}\widetilde{P}\, dt$
is the distribution of the scenery flow of $\mu$.

Recall that the lower Hausdorff dimension of a measure $\tau$ is
\[
\ldim\tau=\inf\{\dim A\,:\,\tau(A)>0\}\]
We note that if $\tau$ has exact dimension $\alpha$ then $\ldim\tau=\alpha$
as well. 

It is simple to verify that $\tau\ll\tau'$ implies $\ldim\tau\geq\ldim\tau'$
and, more generally, if a measure $\tau$ can be written as $\tau=\int\tau_{i}\, d\sigma(i)$,
then $\ldim\tau\geq\essinf_{i\sim\sigma}\ldim\tau_{i}$. Thus if $f$
is a map then $f\tau=\int f\tau_{i}\, d\sigma(i)$, and a similar
bound applies.

We are out to show that \[
\frac{n}{\log a}\mathbb{Z}\subseteq\Sigma_{P_{x}}\subseteq\frac{1}{\log a}\Sigma_{\mu^{(x)}}\cup\frac{n}{\log a}\mathbb{Z}\]
The right hand side follows from the fact that $(\mathcal{M}^{*},P_{x},S^{*})$
is a factor of the $\log b$-suspenssion of $([0,1],\mu,T_{b})$.
To establish the left hand inclusion it suffices to prove the following
theorem.
\begin{thm}
\label{thm:existence-of-spectrum}There exists an $n\in\mathbb{N}$
such that $\frac{n}{\log a}\in\Sigma_{(\mathcal{M}^{*},P,S^{*})}$.\end{thm}
\begin{proof}
Since $\widetilde{P}$ is an ergodic component of $P$ with respect
to $S_{\log a}^{*}$, by Lemma \ref{lem:invariance-of-ergodic-components}
and \ref{lem:ergodicity-vs-spectrum} it suffices to show that $\widetilde{P}$
is not $S^{*}$-invariant. Suppose that it were $S^{*}$-invariant.
We claim that this implies that $\mu$ is Lebesgue, contradicting
the assumption of intermediate dimension.

To this end, choose an integer $d$ such that $d\dim\mu>1$ and write
$\mu^{*d}$ for the $d$-fold convolution of $\mu$, which is the
image of the $d$-fold product $\times_{i=1}^{d}\mu$ by the map $f(x)=\sum_{i=1}^{d}x_{i}$.
We first show that $\mu^{*d}$ has dimension 1. Recall that by Proposition
\ref{pro:prediction-measures-and-their-properties},\[
\mu=\int U_{-\xi(\omega)}\mu_{\omega}\, d\widetilde{\mu}(\omega)\]
Since $\widetilde{P}$ is $S^{*}$-invariant, there is a function
$\xi(x,t)\in[0,1]$ such that \begin{equation}
\mu=\int\int_{0}^{1}(U_{-\xi(\omega,t)}S_{t\log b}\mu_{\omega})|_{[0,1]}\, dt\,\widetilde{\mu}(\omega)\label{eq:mu-as-superposition-of-S-inv-family}\end{equation}
This gives a similar representation of the product measure: write
$t=(t_{1},\ldots,t_{d})$ and $u^{d}$ for uniform measure on $[0,1]^{d},$
and likewise write $\omega=(\omega^{1},\ldots,\omega^{d})$ and $\widetilde{\mu}^{d}=\times_{i=1}^{d}\widetilde{\mu}$.
Then \[
\times_{i=1}^{d}\mu=\int_{\Omega^{d}}\int_{[0,1]^{d}}\times_{i=1}^{d}\left((U_{-\xi(\omega^{i},t^{i})}S_{t_{i}\log b}\mu_{\omega^{i}}^{*})|_{[0,1]}\right)\, du^{d}(t)\, d\widetilde{\mu}^{d}(\omega)\]
Therefore, using the comments preceding the proposition, \[
\ldim\mu^{*d}\geq\essinf_{\omega\sim\widetilde{\mu}^{d}\,,\, t\sim u^{d}}f\left(\times_{i=1}^{d}(U_{-\xi(\omega^{i},t^{i})}S_{t_{i}\log b}\mu_{\omega^{i}}^{*})|_{[0,1]}\right)\]
Since $f$ is linear and $\xi(\omega^{i},t^{i})\in[0,1]$, we have
\[
\geq\essinf_{\omega\sim\widetilde{\mu}^{d}\,,\, t\sim u^{d}}\ldim f\left(\times_{i=1}^{d}(S_{t\log b}\mu_{\omega})|_{[-1,1]}\right)\]
because each of the previous measures is absolutely continuous with
respect to the corresponding measure above. Writing $f_{t}(x)=\sum t_{i}x_{i}$,
we have by another absolute-continuity argument \[
\geq\essinf_{\omega\sim\widetilde{\mu}^{d}}\left(\essinf_{t\sim u^{d}}\ldim f_{t}((\times_{i=1}^{d}\mu_{\omega^{i}})|_{[-b,2b]})\right)\]
Finally, for fixed typical $\omega$ we have \[
\dim\times_{i=1}^{d}\mu_{\omega^{i}}=\sum_{i=1}^{d}\dim\mu_{\omega^{i}}>1\]
so the inner $\essinf$ (over $t\sim u^{d}$) in the previous expression
is $1$ by the following version of Mastrand's classical theorem on
projections of measures. 
\begin{thm}
[Hunt-Kaloshin \cite{HuntKaloshin97}] Let $\sigma$ be an exact dimensional
probability measure on $\mathbb{R}^{d}$ with $\dim\tau=\alpha$.
Then for Lebesgue-a.e. $(t_{1},\ldots,t_{d})\in\mathbb{R}^{d}$, the
image of $\sigma$ under $x\mapsto\sum_{i=1}^{d}t_{i}x_{i}$ is exact
dimensional and has dimension $\min\{1,\dim\sigma\}$. 
\end{thm}
Thus, we have shown that $\mu^{*d}$ has dimension 1. Next, note that
convolution in $\mathbb{R}/\mathbb{Z}$ is obtained by taking the
convolution in $\mathbb{R}$ modulo $1$. Since this is a countable-to-1
local isometry $\mathbb{R}\rightarrow\mathbb{R}/\mathbb{Z}$ it does
not change dimension, so the $d$-th convolution of $\mu$ in $\mathbb{R}/\mathbb{Z}$
has dimension $1$. Since this convolved measure is also $T_{b}$-invariant
it is exact dimensional, and hence its exact dimension is 1, and it
must be Lebesgue measure because this is the only measure of dimension
1 invariant under $T_{b}$. Finally, by examining the Fourier coefficients
and using the elementary relation $\widehat{(\mu^{*d})}(k)=\widehat{\mu}(k)^{d}$,
we conclude that $\mu$ is Lebesgue measure. This is the desired contradiction.
\end{proof}

\subsection{\label{sub:The-phase}The phase }

Let $\mu$ be a $T_{b}$-invariant measure of intermediate dimension
generating a.e. the same distribution $P$. As usual we denote by
$P_{x}$ the distribution generated by $\mu$ at $x$ and by $P_{x,y}$
the distribution generated by $\mu\times\mu$ at $(x,y)$, and by
$\widetilde{P}_{x}$ and $\widetilde{P}_{x,y}$ distributions obtained
from $\widetilde{\mu}$ and $\widetilde{\mu}\times\widetilde{\mu}$
as in Proposition \ref{pro:prediction-measures-and-their-properties}.

Throughout this section $\alpha\in\Sigma_{P}$ and $\varphi$ is the
corresponding eigenvalue. Recall that the phase $p_{\alpha}(P_{x,y})$
is the almost sure value of $\varphi(\sigma)/\varphi(y)$ for $\sigma\times\tau\sim P_{x,y}$,
which is the same as for $\sigma\times\tau\sim\widetilde{P}_{x,y}$.
Therefore for $\mu\times\mu$-typical $(x,y)$ and corresponding typical
$(\omega,\eta)\in\Omega\times\Omega$,\[
p_{\alpha}(P_{x,y})=\frac{\varphi(\mu_{\omega}^{*})}{\varphi(\mu_{\eta}^{*})}\]
Fixing a $\mu$-typical $x_{0}$, the phase measure $\theta_{\alpha}=\theta_{\alpha}(\mu,x_{0})$
is the push-forward of $\mu$ via $y\mapsto p_{\alpha}(P_{x_{0},y})$.
Thus, for $\omega_{0}$ corresponding to $x_{0}$, we find that $\theta_{\alpha}(\mu,x_{0})$
is the push-forward of $\widetilde{\mu}$ through the map $\eta\mapsto\varphi(\mu_{\omega_{0}}^{*})/\varphi(\mu_{\eta}^{*})$. 
\begin{prop}
\label{pro:equidistribution-of-other-phases}If $\lambda\in\Sigma_{P}\setminus\frac{1}{\log b}\mathbb{Q}$
then $\theta_{\lambda}$ is Lebesgue measure.\end{prop}
\begin{proof}
$\eta\mapsto\varphi(\mu_{\eta}^{*})$ is an eigenfunction for the
system $(\Omega,\widetilde{\mu},T)$ with eigenvalue $\lambda\log b$.
Since $\lambda\log b$ is irrational the distribution of $\varphi(\mu_{\eta}^{*})$
is uniform on the circle, so the same is true for $\varphi(\mu_{\omega_{0}}^{*})/\varphi(\mu_{\eta}^{*})$,
and the conclusion follows.\end{proof}
\begin{prop}
\label{pro:description-of-the-phase}If $\alpha=\frac{n}{m\log b}\in\Sigma_{(\mathcal{M}^{*},P,S^{*})}$
and $\mu$ is ergodic then $\theta_{\alpha}$ is uniform measure on
a rotation of the $m$-th roots of unity. In particular if $m=1$
then $\theta_{\alpha}$ consists of a single atom.\end{prop}
\begin{proof}
$\eta\mapsto\varphi(\mu_{\eta}^{*})$ is an eigenfunction of $(\Omega,\widetilde{\mu},T)$
with eigenvalue $(\log b)\alpha=n/m$. The distribution of $\varphi(\mu_{\eta}^{*})$
for $\eta\sim\widetilde{\mu}$ it just the distribution of this eigenfunction,
which is uniform on a rotation of the $m$-th roots of unity. Therefore
the same is true for $\varphi(\mu_{\omega_{0}}^{*})/\varphi(\mu_{\eta}^{*})$,
$\eta\sim\widetilde{\mu}$, proving the proposition.
\end{proof}
Now we turn to the non-ergodic case. For $z=e(t)$ with $t\in[0,1)$
let us denote $L(z)=t/\alpha$.
\begin{lem}
The distribution of $\mu_{\omega}^{*}$, $\omega\sim\widetilde{\mu}$
is the same as the distribution of $S_{L(z)}^{*}\nu$, $(\nu,z)\sim\widetilde{P}_{x_{0}}\times\theta_{\alpha}$,
where $\theta_{\alpha}=\theta_{\alpha}(\mu,x_{0})$. \end{lem}
\begin{proof}
Write for brevity\[
p(\omega)=\frac{\varphi(\mu_{\omega_{0}}^{*})}{\varphi(\mu_{\omega}^{*})}\]
and consider the map $\Omega\rightarrow\mathcal{M}^{*}\times\{|z|=1\}$
defined by \[
\omega\mapsto(S_{-L(p(\omega))}^{*}\mu_{\omega}^{*}\;,\; p(\omega))\]
It suffices to show that this map takes $\widetilde{\mu}$ to $\widetilde{P}\times\theta_{\alpha}$,
and for this we must show that (a) the second component of the image
measure is $\theta_{\alpha}$ and (b) conditioned on the value of
the second component, the distribution of the first component is $\widetilde{P}_{x_{0}}$.

For (a) , for $x\sim\mu$. Then for $\widetilde{\mu}^{(x)}$-a.e.
$\omega$ the value of $p(\omega)=\varphi(\mu_{\omega_{0}}^{*})/\varphi(\mu_{\omega}^{*})$
is $p_{\alpha}(\widetilde{P}_{x_{0},x})$, because, by definition,
$\mu_{\omega_{0}}^{*}\times\mu_{\omega}^{*}$ is a typical element
of $\widetilde{P}_{x,x_{0}}$; and this is the same as $p_{\alpha}(P_{x_{0},x})$,
because $P_{x_{0},x}=\int_{0}^{1}S_{t\log b}^{*}\widetilde{P}_{x_{0},x}\, dt$.
The distribution of $p_{\alpha}(P_{x_{0},x})$ this quantity for $x\sim\mu$
is by definition equal to $\theta_{\alpha}$. 

Next, conditioned on the value $p(\omega)=\varphi(\mu_{\omega_{0}}^{*})/\varphi(\mu_{\omega}^{*})$
we know by Corollary that $S_{-L(z)}^{*}\widetilde{P}_{x}=\widetilde{P}_{x_{0}}$.
This proves the lemma.\end{proof}
\begin{prop}
\label{pro:concentration-of-log-b-phase}If $\alpha\in\Sigma_{P}\cap\frac{1}{\log b}\mathbb{Q}$
then $\theta_{\alpha}$ is singular with respect to Lebesgue.\end{prop}
\begin{proof}
Our strategy is similar to the proof of Theorem \ref{thm:existence-of-spectrum}.
Choose an integer $d$ with $d\cdot\dim\mu>1$; we show that if $\theta_{\alpha}$
were continuous with respect to Lebesgue measure then this would imply
that $\dim\mu^{*d}=1$, which would contradict the to intermediate
entropy of $\mu$, as in Theorem \ref{thm:existence-of-spectrum}. 

We aim to show that $\dim\mu^{*d}=1$. By Proposition \ref{pro:prediction-measures-and-their-properties}
we have\[
\mu=\int(U_{\xi(\omega)}\mu_{\omega})|_{[0,1]}\, d\widetilde{\mu}(\omega)\]
Write $\omega=(\omega^{1},\ldots,\omega^{d})$ and $\widetilde{\mu}^{d}=\times_{i=1}^{d}\widetilde{\mu}$.
Let $f(x)=\sum_{i=1}^{d}x_{i}$. Then \begin{eqnarray*}
\ldim\mu^{*d} & = & \ldim f(\int\times_{i=1}^{d}(U_{\xi(\omega^{i})}\mu_{\omega^{i}})|_{[0,1]}\, d\widetilde{\mu}^{d}(\omega))\\
 & \geq & \essinf_{\omega\sim\widetilde{\mu}^{d}}\ldim f((\times_{i=1}^{d}U_{\xi(\omega^{i})}\mu_{\omega^{i}})|_{[0,1]^{d}})\\
 & \geq & \essinf_{\omega\sim\widetilde{\mu}^{d}}\ldim f((\times_{i=1}^{d}\mu_{\omega^{i}})|_{[-1,1]^{d}})\end{eqnarray*}
where in the last equality we used linearity of $f$. Writing $\nu=(\nu_{1},\ldots,\nu_{d})$,
$z=(z_{1},\ldots,z_{d})$ and $\widetilde{P}^{d}$, $\theta_{\alpha}^{d}$
for the $d$-fold product measures, we can apply the previous lemma
and get\[
=\essinf_{\nu\sim\widetilde{P}_{x_{0}}\,,\, t\sim\theta_{\alpha}^{d}}\ldim f\left((\times_{i=1}^{d}S_{L(z^{i})}^{*}\nu_{i})|_{[-1,1]^{d}}\right)\]
Setting $f_{t}(x)=\sum_{i=1}^{d}e^{t_{i}}x_{i}$ and $L(z)=(L(z_{1}),\ldots,L(z_{d}))$,\[
\geq\essinf_{\nu\sim\widetilde{P}_{x_{0}}\,,\, t\sim\theta_{\alpha}^{d}}\ldim f_{-L(z)}\left((\times_{i=1}^{d}\nu^{i})|_{[-b,b]^{d}}\right)\]
Now with $\nu_{1},\ldots,\nu_{d}$ fixed typical measures for $\widetilde{P}_{x_{0}}$
we know that $\times_{i=1}^{d}\nu_{i}|_{[-b,b]^{d}}$ has exact dimension
$d\dim\mu>1$. Also, since $L$ is a piecewise smooth map, if $\theta_{\alpha}$
were absolutely continuous then the distribution of $L(z)$ for $z\sim\theta_{\alpha}^{d}$
would be absolutely continuous with respect to $d$-dimensional Lebesgue
measure. Hence, applying Marstrand's theorem again, we find that for
$\widetilde{P}_{x_{0}}^{d}$-a.e. choice of $\nu$ the dimension in
the expression above is $1$, so the essential infimum is $1$. This
completes the proof.
\end{proof}
\bibliographystyle{plain}
\bibliography{bib}

\end{document}